\documentclass[a4paper,11pt]{article}
\usepackage{graphicx}
\usepackage{grffile}
\usepackage[latin1]{inputenc}
\usepackage[T1]{fontenc}
\usepackage[english]{babel}
\usepackage{amsfonts}
\usepackage{amssymb,amsmath}
\usepackage{amsthm}
\usepackage{mathtools}
\usepackage{amsxtra}
\usepackage[mathcal,mathscr]{eucal}
\usepackage{float}
\usepackage{setspace}
\newtheorem{teorema}{Theorem}[section]
\newtheorem{lemmaa}[teorema]{Lemma}
\newtheorem{definizione}[teorema]{Definition}
\newtheorem{osservazione}[teorema]{Remark}
\newtheorem{corollario}[teorema]{Corollary}
\newtheorem{proposizione}[teorema]{Proposition}
\newcommand{\N}{\mathbb{N}}

\newcommand{\R}{\mathbb{R}}
\newcommand{\C}{\mathbb{C}}
\newcommand{\PP}{\mathbb{P}}

\newcommand{\Sym}{\emph{Sym}}
\newcommand{\E}{\mathbb{E}}

\title{\bfseries{The number of real eigenvectors of a real polynomial}}
\author{\bfseries{Mauro Maccioni}\thanks{This work has been partially supported by G.N.S.A.G.A. of INDAM and by MIUR}}
\date{29 November 2016}

\begin{document}
\maketitle
\begin{abstract}
I investigate on the number $t$ of real eigenvectors of a real symmetric tensor. In particular, given a homogeneous polynomial $f$ of degree $d$ in $3$ variables, I prove that $t$ is greater or equal than $2c+1$, if $d$ is odd, and $t$ is greater or equal than $\max(3,2c+1)$, if $d$ is even, where $c$ is the number of ovals in the zero locus of $f$. About binary forms, I prove that $t$ is greater or equal than the number of real roots of $f$. Moreover, the above inequalities are sharp for binary forms of any degree and for cubic and quartic ternary forms.
\end{abstract}

\section{Introduction} \label{sez:0}
Given a real homogeneous polynomial $f$ of degree $d$ in $n$ variables, its eigenvectors are $x\in\C^n$ such that $\nabla f(x)=\lambda x$, for some $\lambda\in\C$.\\
In an alternative way, the eigenvectors are the critical points of the Euclidean distance function from $f$ to the
Veronese variety of polynomials of rank one (see \cite{DHOST}).\\
In the quadratic case ($d=2$) the eigenvectors defined in this way coincide with the usual eigenvectors of the symmetric matrix associated
to $f$. By the Spectral Theorem, the eigenvectors of a quadratic polynomial are all real.
So a natural question is to investigate the reality of the eigenvectors of a polynomial $f$ of any degree $d$. The number of complex eigenvectors of a polynomial $f$ of degree $d$ in $n$ variables, when it is finite, is given by
\begin{equation} \label{eq:0}
\left\{
\begin{array}{cc}
((d-1)^n-1)/(d-2), & d\geq3  \\
(d-1)^{n-1}+(d-1)^{n-2}+...+(d-1)^0=n, & d=2
\end{array}
\right.
\end{equation}
The value in this formula has to be counted with multiplicities. The general polynomial has all eigenvectors of multiplicity one. Formula (\ref{eq:0}) appears in \cite{CSS} by Cartwright and Sturmfels.

Our picture is quite complete in the case $n=2$ of binary forms. We show that\\ \\
\textbf{Theorem 1}: \emph{The number of real eigenvectors of a real homogeneous polynomial in $2$ variables is greater or equal than the number of its real roots}.\\ \\
Moreover, we show that the inequality of Theorem 1 is sharp and it is the only essential constraint about the reality of eigenvectors, in the sense that the set of polynomials in $\Sym^d\R^2$ with exactly $k$ real roots contains subsets of positive volume consisting of polynomials with exactly $t$ real eigenvectors, for any $t$ such that $k\leq t\leq d$, $k\equiv t\equiv d\mod2$, $t\geq1$. The congruence mod 2 is an obvious necessary condition on the pair $(k,t)$ which comes from the conjugate action.
Note that all extreme cases are possible, so there are polynomials with maximum number $d$ of real eigenvectors, as well as polynomials with $1$ real eigenvector for odd $d$ (with only one real root by Theorem 1) and polynomials with $2$ real eigenvectors for even $d$ (with zero or two real roots by Theorem 1). There are no polynomials with zero real eigenvectors: this is due to the interpretation of eigenvectors as critical points of the Euclidean distance function, which always attains a real minimum.\\
We can restate the inequality of Theorem 1 by saying that the topological type of $f$ prescribes the possible cases for the number of real eigenvectors.

We investigate the next case $n=3$ of ternary forms. In this case the topological type of $f$ depends on the number of ovals in the real projective plane and on their mutual position (nested or not nested). Again we prove an inequality which follows the same philosophy of Theorem 1. Precisely we have\\ \\
\textbf{Theorem 2}: \emph{Let $t$ be the number of real eigenvectors of a real homogeneous polynomial in $3$ variables with $c$ ovals. Then $t\geq 2c+1$, if $d$ is odd, and $t\geq\max(3,2c+1)$, if $d$ is even}.\\ \\
We give evidence that the inequality of Theorem 2 is sharp, by showing that in the cases $d=3$ and $d=4$, the set of polynomials in $\Sym^d\R^3$
with exactly $c$ real ovals contains subsets of positive volume consisting of polynomials with exactly $t$ real eigenvectors, for any $t$ such that $t$ is odd and $2c+1\leq t\leq7$ ($d=3$) and $\max(3,2c+1)\leq t\leq13$ ($d=4$). Again the condition that $t$ is odd is a necessary condition which follows from the fact that the values in (\ref{eq:0}) are odd for $n=3$ (as for any odd $n$).

The structure of this paper is as follows:\\
In section \ref{sez:1} we give preliminaries and a general result (Lemma \ref{lemma:0}) on the nature of real eigenvectors of a real symmetric tensor.\\
In section \ref{sez:2} we investigate on binary forms. In primis, we give some examples where it is evident that there are some forbidden values for the number of real eigenvectors of a form conditioned to the number of its real roots. Then we give our main Theorem \ref{teo:2}, that shows that the number of real eigenvectors of a real homogeneous polynomial in $2$ variables is greater or equal than the number of its real roots and this inequality is sharp. Theorem 1 follows from this.\\
In section \ref{sez:3} we investigate on ternary forms. In primis, we give some computational examples of ternary cubics where it is evident that there are some forbidden values for $t$ conditioned to $c$, and that all numbers of real eigenvectors are possible for a cubic, according to our main Theorem \ref{teo:4}, that shows that $t$ is greater or equal than $2c+1$, if $d$ is odd, and $t$ is greater or equal than $\max(3,2c+1)$, if $d$ is even. Theorem 2 follows from this. Moreover, we find ternary forms of degree $d$ with a certain number $c$ of ovals, all with the maximum number of real eigenvectors.\\
In section \ref{sez:4} we give examples of cubics and quartics with the minimum and the maximum number of real eigenvectors in all possible topological cases, showing that for $d=3,4$ the inequality of Theorem \ref{teo:4} is again sharp (Propositions \ref{prop:99} and \ref{prop:100}).\\
In section \ref{sez:5}, we give some computational examples of ternary quartics and sextics with all possible values of $t$ conditioned to the value of $c$ in some topological cases.\\
The results in this paper are part of my doctoral thesis at the University of Firenze, with advisor Giorgio Ottaviani.

\section{Preliminaries} \label{sez:1}

\begin{definizione}(\cite{CSS},\cite{L},\cite{Q}) \label{def:1}
Let $x\in\C^n$ be and let $A=(a_{i_1,i_2,...,i_d})$ be a symmetric tensor of order $d$ and dimension $n$. We define $Ax^{d-1}$ to be the vector in $\C^n$ whose $j$-th coordinate is the scalar $$(Ax^{d-1})_j=\sum_{i_2=1}^{n}\cdots\sum_{i_d=1}^{n}a_{j,i_2,...,i_d}x_{i_2}\cdots x_{i_d}$$ Then, if $\lambda\in\C$ and $\tilde{x}\in\C^n\setminus\left\{0\right\}$ are elements such that $A{\tilde{x}}^{d-1}=\lambda\tilde{x}$, we say that $\lambda$ is an eigenvalue of $A$ and $\tilde{x}$ is an eigenvector of $A$.
\end{definizione}

\begin{osservazione}(\cite{CSS},\cite{EL}) \em \label{oss:1}
Consider $f(x)\equiv f(x_1,...,x_n)$ the homogeneous polynomial in $\C[x_1,...,x_n]$ of degree $d$ associated to the symmetric tensor $A$ by the relation $$ f(x_1,...,x_n)=A\cdot x^d=\sum_{i_1}^{n}\cdots\sum_{i_d}^{n}a_{i_1,i_2,...,i_d}x_{i_1}\cdots x_{i_d}=x\cdot Ax^{d-1}$$ Then $\tilde{x}\in\C^n$ is an eigenvector of $A$ with eigenvalue $\lambda\in\C$ if and only if $$\nabla f(\tilde{x})=\lambda\tilde{x}$$ Moreover, the eigenvectors of $A$ are precisely the fixed points of the projective map $$ \nabla f:\PP^{n-1}(\C)\longrightarrow\PP^{n-1}(\C), \; [x]\longmapsto[\nabla f(x)] $$ which is well-defined
provided that the hypersurface $\left\{f=0\right\}$ has no singular points. 
Finally, all previous characterizations are equivalent to saying that $\tilde{x}\in\C^n$ is an eigenvector of $A$ if and only if all the $2\times2$ minors of the $2\times n$ matrix
$$
\left(
\begin{array}{cccc}
f_{x_1}(\tilde{x}) & f_{x_2}(\tilde{x}) & \ldots & f_{x_n}(\tilde{x}) \\
x_1 & x_2 & \ldots & x_n
\end{array}
\right)
$$
vanish on $\tilde{x}$, or obviously that the vectors $\nabla f(\tilde{x})$ and $\tilde{x}$ are proportional.
\end{osservazione}



\begin{teorema} \label{teo:1}
Every symmetric tensor has at most
$$
\left\{
\begin{array}{cc}
((d-1)^n-1)/(d-2), & d\geq3  \\
(d-1)^{n-1}+(d-1)^{n-2}+...+(d-1)^0=n, & d=2
\end{array}
\right.
$$ 
distinct normalized eigenvalues. This bound is attained for generic symmetric tensors.
\end{teorema}
The above Theorem is a result by D. Cartwright and B. Sturmfels in \cite{CSS} (Theorem 5.5), although in \cite{ASS} it has been remarked that it was already known by Fornaess and Sibony (\cite{FS}) in the setting of dynamical systems.

Our goal is study the number of real eigenvectors of $f$, supposing that $\left\{f=0\right\}$ has a certain number of real connected components.

\begin{lemmaa} \label{lemma:0}
A vector $v\in\R^n$ is a real eigenvector of $f\in\Sym^d(\R^n)$ if and only if $v$ is a critical point of $f|_{S^{n-1}}$, where $S^{n-1}=\left\{x\in\R^n\, | \, \|x\|=1\right\}$.
\end{lemmaa}

\begin{proof}
By Remark \ref{oss:1}, finding (real) eigenvectors of $f$ is equivalent to finding (real) fixed points of the projective application $\nabla f$, or also to solving the system
$$
Sys_1=\left\{
\begin{array}{c}
f_{x_1}(x_1,x_2,\ldots,x_n)-\lambda x_1=0 \\
f_{x_2}(x_1,x_2,\ldots,x_n)-\lambda x_2=0 \\
\vdots \\
f_{x_n}(x_1,x_2,\ldots,x_n)-\lambda x_n=0
\end{array}
\right.
$$
with $\lambda\in\C$ ($\lambda\in\R$).\\
Consider the Lagrangian map 
$$L(x_1,x_2,\ldots,x_n,\lambda)=f(x_1,x_2,\ldots,x_n)-\lambda g(x_1,x_2,\ldots,x_n)$$
where $g(x_1,x_2,\ldots,x_n)=x_1^2+x_2^2+\ldots+x_n^2-1$. Then, the solutions of system
$$
Sys_2=\left\{
\begin{array}{c}
L_{x_1}(x_1,x_2,\ldots,x_n,\lambda)\equiv f_{x_1}(x_1,x_2,\ldots,x_n)-\lambda x_1=0 \\
L_{x_2}(x_1,x_2,\ldots,x_n,\lambda)\equiv f_{x_2}(x_1,x_2,\ldots,x_n)-\lambda x_2=0 \\
\vdots \\
L_{x_n}(x_1,x_2,\ldots,x_n,\lambda)\equiv f_{x_n}(x_1,x_2,\ldots,x_n)-\lambda x_n=0 \\
L_{\lambda}(x_1,x_2,\ldots,x_n,\lambda)\equiv g(x_1,x_2,\ldots,x_n)=0
\end{array}
\right.
$$
are all solutions of $Sys_1$; but solving $Sys_2$ gives critical points $v=(v_1,v_2,\ldots,v_n,\lambda_0)$ of $L$, that is critical points of $f|_{S^{n-1}}$ (it is the method of Lagrange multipliers), that is the solutions of the system\\ $Sys_3=\nabla(f|_{S^{n-1}})(x_1,x_2,\ldots,x_n)=(0,0,\ldots,0)$.
\end{proof}

\section{Binary forms} \label{sez:2}

Let $f\in\Sym^d(\R^2)$ be a binary form, that is a homogeneous polynomial of degree $d$ in two variables $x$, $y$. In this case, the question of the number of real eigenvectors of $f$ in relation with the number of real connected components of $\left\{f=0\right\}$ simply means that we must compare the real roots of $f$ with the real roots of the discriminant $yf_x-xf_y$ (also known as critical real roots of $f$) of the matrix  
$$
\left(
\begin{array}{cc}
f_{x}(x,y) & f_{y}(x,y) \\
x & y
\end{array}
\right).
$$

\begin{osservazione} \label{oss:binarie} \em
Let $f$ be a binary form. For a sample of $100000$ forms $f$ of degree $4$, $5$, where $$f=\sum_{i=0}^d\sqrt{\left(
\begin{array}{c}
d  \\
i
\end{array}
\right)}a_ix^{d-i}y^i, \; a_i\approx N(0,1)$$ and $N(0,1)$ is the normal distriubution of mean $0$ and variance $1$, we have estimated the probabilities for the variable $t$ conditioned to the values of $q$, where $q$ is the number of real roots of $f$ and $t$ is the number of real roots of $yf_x-xf_y$:
\begin{table}[H]
\begin{tabular}{||p{2cm}||*{3}{c|}|} 
\hline
$q$ & $t=0$ & $t=2$ & $t=4$ \\
\hline
\bfseries $4$ & $0$ & $0$ & $1$ \\
\hline
\bfseries $2$ & $0$ & $0.5160$ & $0.4840$ \\
\hline
\bfseries $0$ & $0$ & $0.3038$ & $0.6962$ \\
\hline 
\end{tabular}
\caption{$d=4$}
\end{table}
\begin{table}[H]
\begin{tabular}{||p{2cm}||*{3}{c|}|} 
\hline
$q$ & $t=1$ & $t=3$ & $t=5$ \\
\hline
\bfseries $5$ & $0$ & $0$ & $1$ \\
\hline
\bfseries $3$ & $0$ & $0.7186$ & $0.2814$ \\
\hline
\bfseries $1$ & $0.0516$ & $0.6234$ & $0.3250$ \\
\hline 
\end{tabular}
\caption{$d=5$}
\end{table} 
Hence, we note that there are some prohibited values of $t$ in relation to the value of $q$, that is we can guess that $q\leq t$ and this is the only constraint for $q\geq1$.\\
Again for a sample of $100000$ forms $f$, we have estimated the probabilities of the aleatory variables $X_f=(0,2,4)$ for $d=4$, $Y_f=(1,3,5)$ for $d=5$ and respectively $X_{yf_x-xf_y}=(0,2,4)$, $Y_{yf_x-xf_y}=(1,3,5)$ with respect to $f$ and $yf_x-xf_y$ and then relative expected values and we expect that $\E(X_f)\approx\sqrt{d}$ and $\E(X_{yf_x-xf_y})\approx\sqrt{3d-2}$ and the same for $\E(Y_f)$ and $\E(Y_{yf_x-xf_y})$ (see Example 1.6 in \cite{DH} and Example 4.8 in \cite{DHOST}):
\begin{table}[H]
\begin{tabular}{||p{3cm}||*{3}{c|}|} 
\hline
$X_f$ & $0$ & $2$ & $4$ \\
\hline
\bfseries $\approx$ probability & $0.1350$ & $0.7307$ & $0.1342$ \\
\hline
\end{tabular}
\caption{$d=4$}
\end{table}
\begin{table}[H]
\begin{tabular}{||p{3cm}||*{3}{c|}|}
\hline
$Y_f$ & $1$ & $3$ & $5$ \\
\hline
\bfseries $\approx$ probability & $0.4167$ & $0.5491$ & $0.0343$ \\
\hline
\end{tabular}
\caption{$d=5$}
\end{table}
whence $\E(X_f)=1.9984\approx\sqrt{4}=2$ and $\E(Y_f)=2.2352\approx\sqrt{5}$.
\begin{table}[H]
\begin{tabular}{||p{3cm}||*{3}{c|}|}
\hline
$X_{yf_x-xf_y}$ & $0$ & $2$ & $4$ \\
\hline
\bfseries $\approx$ probability & $0$ & $0.4190$ & $0.5810$ \\
\hline
\end{tabular}
\caption{$d=4$}
\end{table}
\begin{table}[H]
\begin{tabular}{||p{3cm}||*{3}{c|}|}
\hline
$Y_{yf_x-xf_y}$ & $1$ & $3$ & $5$ \\
\hline
\bfseries $\approx$ probability & $0.0224$ & $0.6569$ & $0.3207$ \\
\hline
\end{tabular}
\caption{$d=5$}
\end{table}
whence $\E(X_{yf_x-xf_y})=3.1620\approx\sqrt{10}$ and $\E(Y_{yf_x-xf_y})=3.5966\approx\sqrt{13}$.\\
We have the following
\end{osservazione}

\begin{teorema} \label{teo:2}
Let $f\in\Sym^d(\R^2)$ be, with $d\in\N$. Then $\max(\# real \; roots \; of \; f,1)\leq\# real \; eigenvectors \; of \; f$ and this relation is the only constraint for the number $q$ of real roots of $f$, in the sense that for any pair $(q,t)$ such that $q\equiv t\equiv d\mod2$ and $\max{(q,1)}\leq t\leq d$ the set $$\left\{f\in\Sym^d(\R^2)\,|\, \# real \; roots \; of \; f=q,\, \# real \; eigenvectors \; of \; f=t\right\}$$ has positive volume.
\end{teorema}

\begin{proof}
Let $q$ be the number of real roots of $f$. If $q=0$, the thesis follows immediately; therefore, consider $q\geq1$.\\ 
There are $q$ lines through the origin of $\R^2$ corresponding to the $q$ roots of $f$ and each of these lines meets the circle $x^2+y^2=1$ in two real points, that is in $2q$ total real points. Consider the following parametrization of the circle 
$$S^1:\left\{
\begin{array}{c}
x=\cos\theta \\
y=\sin\theta
\end{array}\right., \; \theta\in[0,2\pi) 
$$
and the function $F(\theta)=f(\cos\theta,\sin\theta)$, that is $F$ is the restriction of $f$ on $S^1$; evidently, the real roots of $F$ ($2q$) are twice the real roots of $f$ ($q$), or for each real root of $f$ in $\PP(\R^2)$, we have a uniquely determined pair of real roots of $F$. In particular, if for a given $\bar{\theta}$ we have $F(\bar{\theta})=0$, then $F(\bar{\theta}+\pi)=0$ and the line through the points $(\cos\bar{\theta},\sin\bar{\theta})$, $(\cos(\bar{\theta}+\pi),\sin(\bar{\theta}+\pi))=(-\cos\theta,-\sin\theta)$ corresponds to a real root of $f$ in $\PP(\R^2)$ and conversely. Now consider $F^{\prime}(\theta)=-\sin\theta f_x(\cos\theta,\sin\theta)+\cos\theta f_y(\cos\theta,\sin\theta)$. By Rolle's Theorem, between two real roots of $F$ there exists at least one real root of $F^{\prime}$ and then $F^{\prime}$ has at least $2q$ real roots. Consider $G(\theta)=g(\cos\theta,\sin\theta)$, where $g=-yf_x+xf_y$, that is $G$ is the restriction of the polynomial $g$ on $S^1$; then obviously $G(\theta)=F^{\prime}(\theta)$, hence $G$ has at least $2q$ real roots and therefore $g$ has at least $q$ real roots. We get $t\geq q$ as we wanted.\\
Finally, we must prove the following: $$ \forall n\in\N_0, \; \forall h\in\left\{h\in\N_0\,|\, h=2n\right\}, \; \exists f\in\Sym^d(\R^2) \; s.t. \; q=n, \; t=n+h$$ It is sufficient to consider binary forms of even degree $t$ as Fourier polynomials $$g(\cos\theta, \sin\theta)=\left(1+\frac{\cos(2\theta)}{2}\right)+s(\cos(t\theta)+\sin(t\theta))$$ and binary forms of odd degree as Fourier polynomials $$g(\cos\theta, \sin\theta)=\cos(\theta)\left(\left(1+\frac{\cos(2\theta)}{2}\right)+s(\cos(t\theta)+\sin(t\theta))\right)$$ where $s\in\R$. Then we can choose $s$ such that the corresponding Fourier polynomial $g$ of degree $t$ has $q$ real roots in $[0,\pi)$ and its derivative with respect to $\theta$ has exactly $t$ real roots in $[0,\pi)$ (see Figures \ref{fig:1}, \ref{fig:2}, \ref{fig:3}); hence, taking $f=g(x^2+y^2)^{\frac{d}{2}-\frac{t}{2}}$, we have a polynomial $f$ of degree $d$ with exactly $q$ real roots and $t$ real eigenvectors.
\end{proof}
\begin{figure}[H]
	\centering
		\includegraphics[width=10cm,height=4cm]{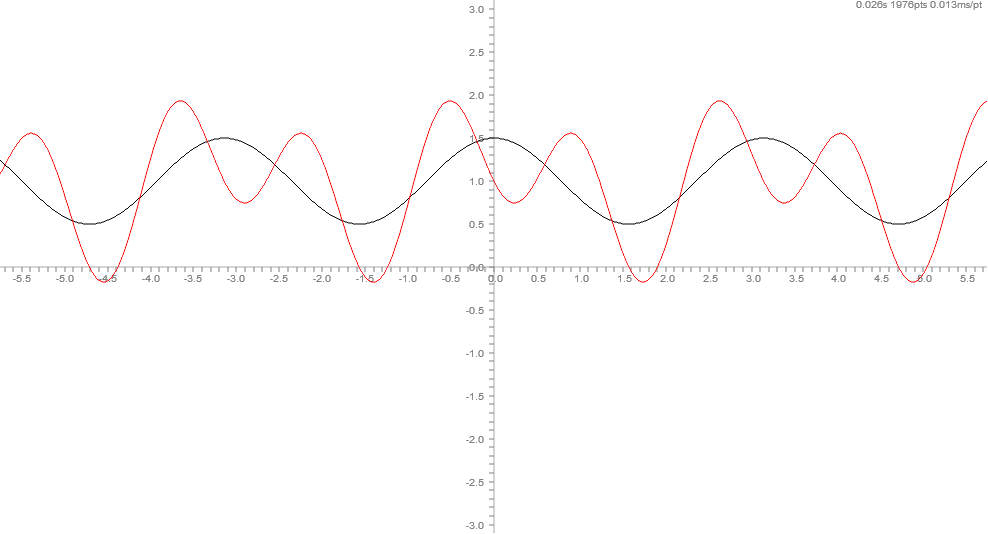}
		\caption{The two graphics of $g$ respectively for $s=0$ (the central one) and $s=-\frac{1}{2}$ (its perturbation). The second one has $q=2$ real roots and its derivative has $t=4$ real roots}
	\label{fig:1}
\end{figure}
\begin{figure}[H]
	\centering
		\includegraphics[width=10cm,height=6cm]{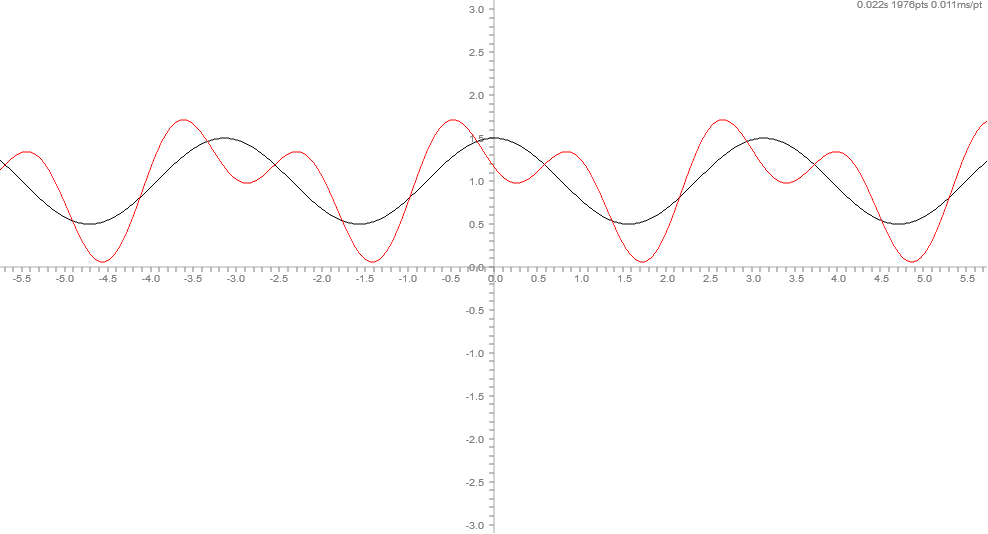}
		\caption{The two graphics of $g$ respectively for $s=0$ (the central one) and $s=-\frac{1}{3}$ (its perturbation). The second one has $q=0$ real roots and its derivative has $t=4$ real roots}
	\label{fig:2}
\end{figure}
\begin{figure}[H]
	\centering
		\includegraphics[width=10cm,height=6cm]{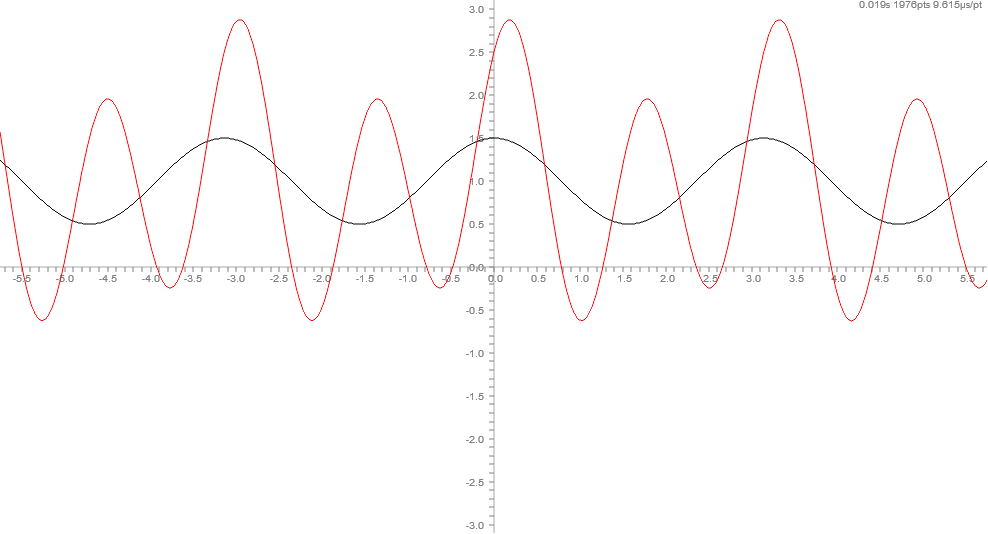}
		\caption{The two graphics of $g$ respectively for $s=0$ (the central one) and $s=2$ (its perturbation). The second one has $q=4$ real roots and its derivative has $t=4$ real roots}
	\label{fig:3}
\end{figure}

\begin{corollario} \label{cor:1}
If $f$ of degree $d$ has exactly $d$ real roots, then $f$ has exactly $d$ real eigenvectors.
\end{corollario}

Corollary \ref{cor:1} is found also in \cite{ASS} by H. Abo, A. Seigal and B. Sturmfels in Remark 6.7, as a consequence of Corollary 6.5. 





\section{Ternary forms} \label{sez:3}

\begin{osservazione} \label{oss:aggiuntiva23} \em
Let $f\in\Sym^d(\R^3)$ be a ternary form, that is $f$ is a homogeneous polynomial of degree $d$ in three variables $x$, $y$, $z$. Then $\left\{f=0\right\}$ has at most $\frac{(d-1)(d-2)}{2}+1$ real connected components in $\PP(\R^3)$ and, by Theorem \ref{teo:1}, $f$ has $((d-1)^3-1)/(d-2)=(d-1)^2+(d-1)+1$ distinct eigenvectors in the general case (note that the number $(d-1)^2+(d-1)+1$ is odd, $\forall\,d\in\N$). By Proposition 11.6.1 in \cite{CS}, if $d$ is odd, $\left\{f=0\right\}$ has a finite number $c+1$ of connected components in $\PP(\R^3)$, $c$ ovals and one pseudo-line. Then the complement $S^2\setminus\left\{f=0\right\}$ consists of $2c+2$ connected components (regions) which are symmetric in pairs. $f$ has constant sign on each region and the signs are opposite for symmetric regions. Again by Proposition 11.6.1 in \cite{CS}, if $d$ is even, $\left\{f=0\right\}$ has only a finite number $c$ of connected components in $\PP(\R^3)$, all ovals. Then the complement $S^2\setminus\left\{f=0\right\}$ consists of $2c+1$ connected components (regions), $2c$ of them are symmetric in pairs. Again $f$ has constant sign on each region and the sign is the same for symmetric regions.
\end{osservazione}

\begin{teorema}[Harnack's curve](\cite{CS}) \label{teo:har}
For any algebraic curve of degree $d$ in the real projective plane, the number of connected components $w$ is bounded by $$\frac{1-(-1)^d}{2}\leq w \leq\frac{(d-1)(d-2)}{2}+1$$ The maximum number is one more than the maximum genus of a curve of degree $d$ and it is attained when the curve is nonsingular. Moreover, any number of components in this range can be attained.
\end{teorema}

\begin{definizione} \label{def:M-curve}
A curve which attains the maximum number of real connected components is called an $M$-curve.
\end{definizione}

\begin{osservazione} \em \label{oss:2}
Given a sample of real ternary cubic forms $f$, we can compute eigenvectors of $f$ with Macaulay2, since the eigenvectors of $f$ are the solutions of the system associated to the ideal $I=(yf_x-xf_y,zf_y-yf_z,zf_x-xf_z)$ (Remark \ref{oss:1}), then we can compute them by the Eigenvectors Method (see Stickelberger's Theorem and Theorem 4.23 in \cite{ACO}).\\
For a sample of $1000$ real ternary cubic forms $f$, where $$f=\sum_{j_0+j_1+j_2=3}\sqrt{\left(
\begin{array}{c}
3  \\
j_0 \, j_1 \, j_2
\end{array}
\right)}a_{j_0j_1j_2}x_0^{j_0}x_1^{j_1}x_2^{j_2}, \; a_i\approx N(0,1)$$ with $c$ ovals, we have estimated the probabilities for the variable $t$ conditioned to variable $c$ in the following table:
\begin{table}[H]
\begin{tabular}{||p{2cm}||*{4}{c|}|} 
\hline
$t$ & $1$ & $3$ & $5$ & $7$ \\
\hline
\bfseries $c=1$ & $0$ & $0,026$ & $0,51$ & $0,464$ \\
\hline
\bfseries $c=0$ & $0,038$ & $0,186$ & $0,422$ & $0,354$ \\
\hline
\end{tabular}
\caption{$d=3$}
\label{table:1}
\end{table}
where $t$ is the number of real eigenvectors of $f$; given $\Delta(f)$ the discriminant of degree $12$ of $f$ (see Proposition 4.4.7 pag. 167, Example 4.5.3 pag. 171 and Formula (4.5.8) pag. 173 in \cite{BBS}), in particular, if $\Delta(f)>0$ then $f$ has two components ($c=1$), while if $\Delta(f)<0$ one ($c=0$).\\
Again for a sample of $1000$ ternary cubic forms $f$, we have estimated the probabilities of aleatory variables $X=(0,1)$, $Y=(1,3,5,7)$ and then their relative expected values and we expect that $\E(Y)\approx1+\frac{8}{7}\sqrt{14}\approx5,276$ (see \cite{DH}, the last Table in subsection 5.2):
\begin{table}[H]
\begin{tabular}{||p{3cm}||*{2}{c|}|}
\hline
$X$ & $0$ & $1$  \\
\hline
\bfseries $\approx$ probability & $0,735$ & $0,265$ \\
\hline
\end{tabular}
\caption{$d=3$}
\end{table}
whence $\E(X)=0,265$.
\begin{table}[H]
\begin{tabular}{||p{3cm}||*{4}{c|}|}
\hline
$Y$ & $1$ & $3$ & $5$ & $7$ \\
\hline
\bfseries $\approx$ probability & $0,028$ & $0,144$ & $0,445$ & $0,383$ \\
\hline
\end{tabular}
\caption{$d=3$}
\end{table}
whence $\E(Y)=5.366\approx5.276$.
We have the following result
\end{osservazione}

\begin{teorema} \label{teo:4}
Let $f$ be a ternary form of degree $d$ and suppose that $f$ has $c$ ovals. Then, if $d$ is odd, we have $2c+1\leq\# real \; eigenvectors \; of \; f$ and if $d$ is even, we have $\max{(2c+1,3)}\leq\# real \; eigenvectors \; of \; f$.
\end{teorema}

\begin{proof}
By Lemma \ref{lemma:0}, finding real eigenvectors of $f$ means finding classes $[(x_0,y_0,z_0)]\in\PP(\R^3)$ such that $(x_0,y_0,z_0)\in S^2$ is a critical point of $f$ on the sphere, that is a maximum, minimum or saddle point of $f$ on $S^2$. By Remark \ref{oss:aggiuntiva23}, we have that the complement $S^2\setminus\left\{f=0\right\}$ is divided at least into $2c$ pairs of symmetric regions, in which $f$ has constant sign and $f$ attains a non zero maximum inside any region where $f$ is positive, and a non zero minimum inside any region where $f$ is negative. Then, for any non zero maximum $v$ there is an antipodal $-v$ which is a non zero minimum if $f$ has odd degree, while for any non zero maximum (minimum) $v$ there is an antipodal $-v$ which is a non zero maximum (minimum) if $f$ has even degree; in conclusion, we have at least $2c$ critical points on the sphere corresponding to maxima or minima of $f$ and then $f$ has at least $c$ real eigenvectors. Consider now the following situations:
\begin{enumerate}
	\item $f\in\Sym^d(\R^3)$, $d$ odd. In this case, by Remark \ref{oss:aggiuntiva23} there are $2c+2$ regions on the sphere, then $2c+2$ total maxima and minima and hence $f$ has at least $c+1$ real eigenvectors.
	\item $f\in\Sym^d(\R^3)$, $d$ even. In this case, by Remark \ref{oss:aggiuntiva23} there are $2c+1$ regions on the sphere, then $2c+2$ total maxima and minima and hence $f$ has at least $c$ real eigenvectors and at least another one, given by a non zero maximum (minimum) $v$ and by its antipodal $-v$ which is a non zero maximum (minimum) of $f$ in the internal of the complement on $S^2$ of the union of all other $2c$ symmetric regions, that is $f$ has at least $c+1$ real eigenvectors.
\end{enumerate}
We must consider also the saddle points of $f$ on $S^2$. By Morse's equation (see Theorem 5.2 pag. 29 in \cite{JWM2}) \begin{equation}\label{eq:1} \sum_{\gamma}(-1)^{\gamma}C_{\gamma}=\chi(S^2) \end{equation} where $\gamma\in\left\{0,1,2\right\}$ is the index of critical points of $f$ on $S^2$ (respectively, we have a maximum, saddle or minimum point if $\gamma$ is $0$, $1$ or $2$), $C_{\gamma}$ is the number of critical points with index $\gamma$ of $f|_{S^2}$ and $\chi(S^2)=2$ is the Euler's characteristic of $S^2$, we have the following equation: $$C_0-C_1+C_2=2$$ 
We have seen that if $f$ has $c$ ovals we have at least $2c+2$ total maxima and minima of $f$ on $S^2$ and then $$C_0+C_2=C_1+2\geq 2c+2\Longrightarrow C_1\geq2c$$
Hence, the total number of critical points of $f$ on the sphere is at least $2c+2+2c=4c+2$ and then $f$ has at least $2c+1$ real eigenvectors.\\
Finally, note that if $d$ is even and if $c=0$, by Weierstrass's Theorem we have that $f$ attains at least a pair of absolute maxima and a pair of absolute minima on $S^2$, then $f$ has at least $2$ real eigenvectors, hence $3$ because the total number of eigenvectors of $f$ is always odd and therefore, if $d$ is even, $f$ has at least $\max\left\{2c+1,3\right\}$ real eigenvectors.
\end{proof}

\begin{osservazione} \label{oss:milnor} \em
Equation (\ref{eq:1}) can be seen in an equivalent way as a consequence of Poincaré-Hopf's Theorem as in \cite{JWM}, pag. 35.
\end{osservazione}

\begin{corollario} \label{cor:cubiche}
Consider $f\in\Sym^3(\R^3)$. Then, according to Remark \ref{oss:2}, if $f$ has two components it has at least three real eigenvectors (see Figure \ref{fig:79}).
\end{corollario}

\begin{figure}[H]
	\centering
		\includegraphics[width=10cm,height=4cm]{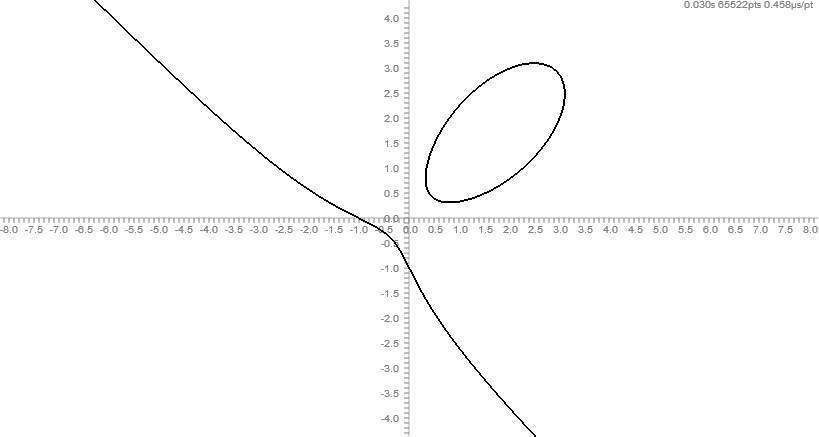}
		\caption{$x^3+y^3+1+6\lambda xy=0$, $\lambda>-\frac{1}{2}$}
	\label{fig:79}
\end{figure}

\begin{osservazione} \label{oss:3} \em
For an $M$-curve we have the following:
\begin{enumerate}
	\item $f\in\Sym^d(\R^3)$, $d$ odd. In this case, by Theorem \ref{teo:har} we have that an $M$-curve has $\frac{(d-1)(d-2)}{2}+1$ components, $\frac{(d-1)(d-2)}{2}$ ovals and one pseudo-line and then by Theorem \ref{teo:4} $f$ has at least $(d-1)(d-2)+1=d^2-3d+3$ real eigenvectors.
	\item $f\in\Sym^d(\R^3)$, $d$ even. In this case, by Theorem \ref{teo:har} we have that an $M$-curve has $\frac{(d-1)(d-2)}{2}+1$ components, all ovals and then by Theorem \ref{teo:4} $f$ has at least $(d-1)(d-2)+3=d^2-3d+5$ real distinct eigenvectors.
\end{enumerate}
\end{osservazione}

\begin{osservazione} \label{oss:connessione} \em
Having fixed the topological type of a form $f\in\Sym^d(\R^3)$, $d=3,4$, i.e. having fixed the number $c$ of ovals of $f$, the set of all forms such that they have the same number $c$ of $f$ is connected (see Theorem 1.7 in \cite{CO}).
\end{osservazione}

\begin{osservazione} \label{oss:drette} \em
Consider a form $f\in\Sym^d(\R^3)$ such that $f=l_1l_2\cdots l_d$, where $l_i$ are linear ternary forms, that is $f$ is a singular form of degree $d$ such that its real locus of zeros consists of $d$ lines in $\R^2$. If we choose all $l_i$ such that $\forall i:1,\ldots, d$ the set $\left\{l_i=0\right\}\cap\left(\cup_{i\neq j}\left\{l_j=0\right\}\right)$ consists of $d-1$ distinct points $P_{i,j}$ in $\R^2$, i.e. each line meets all the others in $d-1$ distinct points, $f$ has always the maximum number $t$ of real eigenvectors with multiplicity $1$. Then, we can perturb $f$ by $\epsilon g$, $g\in\Sym^d\R^3$, $\epsilon\in\R_+$ small enough and obtain a nonsingular quartic, smooth in $P_{i,j}$ depending on the sign of $g$ in $P_{i,j}$, with the maximum $t$. These results are in \cite{ASS}, precisely see Theorem 6.1 and Corollary 6.2.
\end{osservazione}

\section{The inequalities of Theorem \ref{teo:4} are sharp for ternary cubics and quartics} \label{sez:4}

\begin{proposizione} \label{prop:99}
Let $c\in\left\{0,1\right\}$ and let $t$ be odd such that $2c+1\leq t\leq7$. Then the set $$\left\{f\in\Sym^3(\R^3)\,|\, f \; has \; c \; ovals,\, \# real \; eigenvectors \; of \; f=t\right\}$$ has positive volume.
\end{proposizione}

\begin{proof}
By Remark \ref{oss:connessione}, we must show examples of ternary cubic forms such that $c\in\left\{0,1\right\}$ and $t$ attains the maximum and the minimum value. We have the following examples:
\begin{itemize}
	\item $t$ maximum. By Remark \ref{oss:drette}, we can take $f=xy(x+y+1)$, $\epsilon=\frac{1}{1000}$, $g_1=x^3+y^3-2$ and $g_2=-x^3-y^3+2$ to obtain $f_1=f+\epsilon g_1$ and $f_2=f+\epsilon g_2$ with, respectively, $1$ and $0$ ovals and $7$ real eigenvectors (see Figures \ref{fig:92}, \ref{fig:93}, \ref{fig:94}).
	\item $t$ minimum. Then we have:
\begin{itemize}
	\item $f$ has $0$ ovals. In this case, we can find the Weierstrass form \onehalfspacing{$f=y^2-x^3-\frac{1}{9}x^2-x-1$} (see Figure \ref{fig:95}) with $1$ real eigenvector.
	\item $f$ has $1$ oval. In this case, we can find the Weierstrass form \onehalfspacing{$f=y^2-\frac{2}{100}x^3+\frac{45}{100}x^2+\frac{303}{100}x+\frac{29}{100}$} (see Figure \ref{fig:96}) with $3$ real eigenvectors.
\end{itemize}
\end{itemize}
\end{proof}

\begin{figure}[H]
	\centering
		\includegraphics[width=10cm,height=4cm]{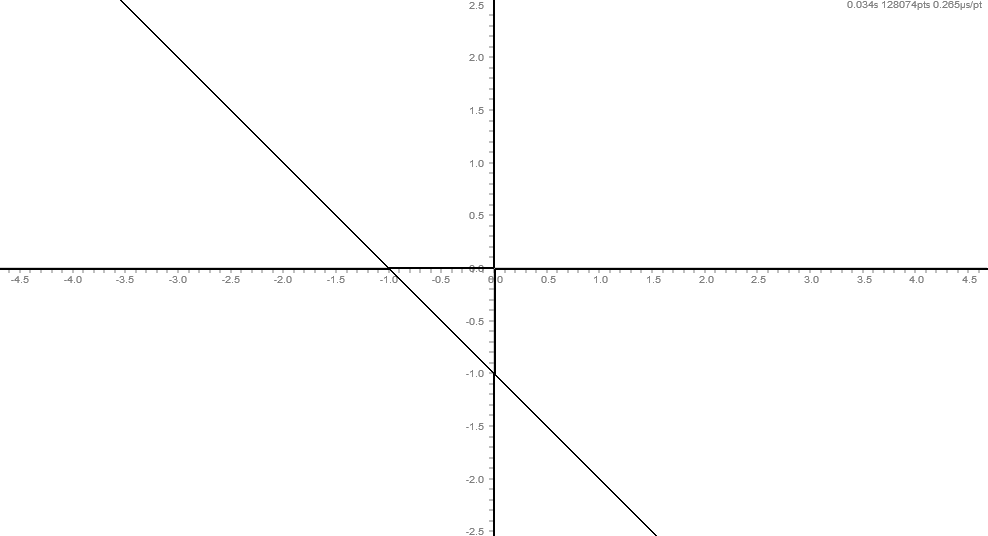}
		\caption{$d=3$, $f=xy(x+y+1)$}
	\label{fig:92}
\end{figure}	
\begin{figure}[H]
	\centering
		\includegraphics[width=10cm,height=4cm]{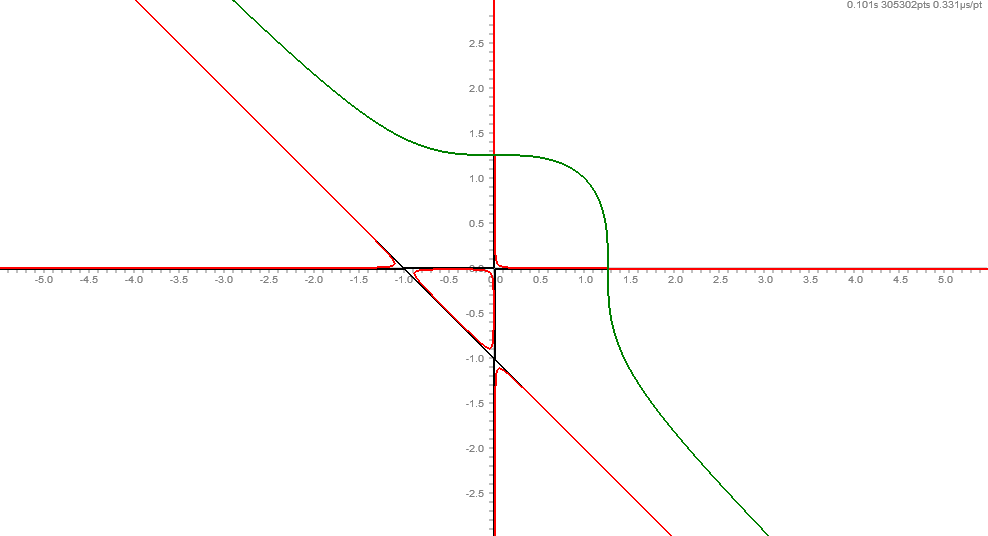}
		\caption{$d=3$, $f=xy(x+y+1)$, $g_1=x^3+y^3-2$ which is negative on the three singular points of $f$, $f_1=f+\frac{1}{1000}g_1$ which has $1$ oval}
	\label{fig:93}
\end{figure}
\begin{figure}[H]
	\centering
		\includegraphics[width=10cm,height=6cm]{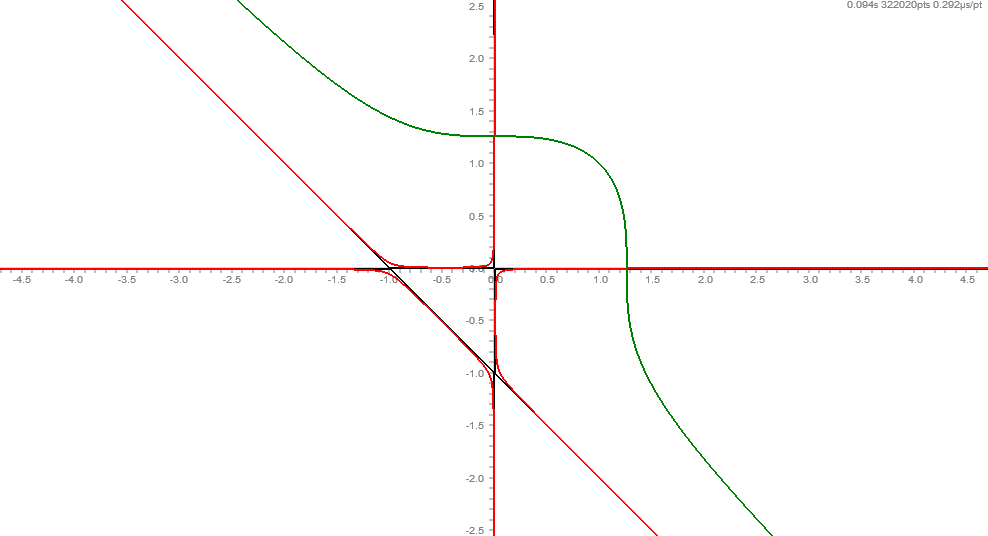}
		\caption{$d=3$, $f=xy(x+y+1)$, $g_2=-x^3-y^3+2$ which is positive on the three singular points of $f$, $f_2=f+\frac{1}{1000}g_2$ which has $0$ ovals}
	\label{fig:94}
\end{figure}
\begin{figure}[H]
	\centering
		\includegraphics[width=10cm,height=6cm]{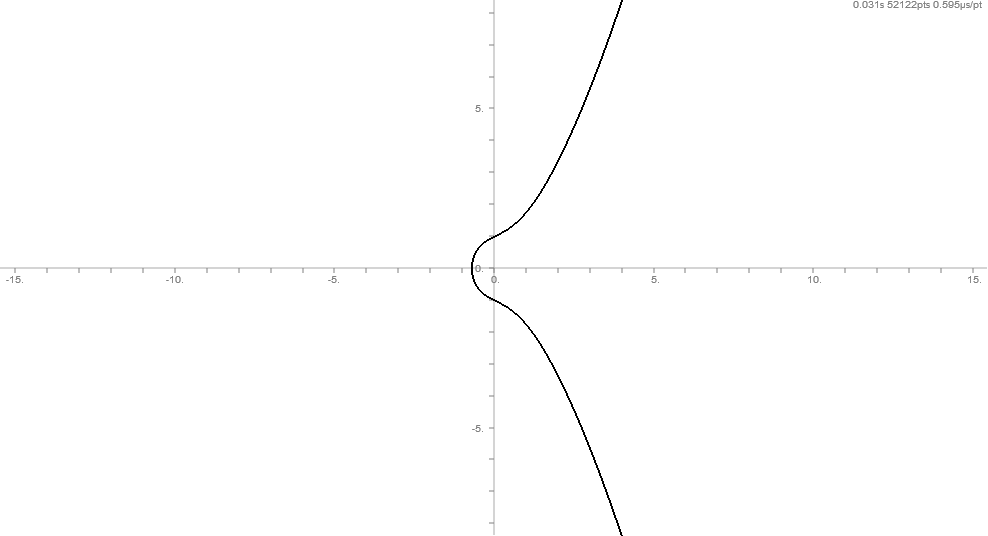}
		\caption{$d=3$, $f=y^2-x^3-\frac{1}{9}x^2-x-1$ which has $c=0$ ovals and $t=1$ real eigenvector}
	\label{fig:95}
\end{figure}
\begin{figure}[H]
	\centering
		\includegraphics[width=10cm,height=6cm]{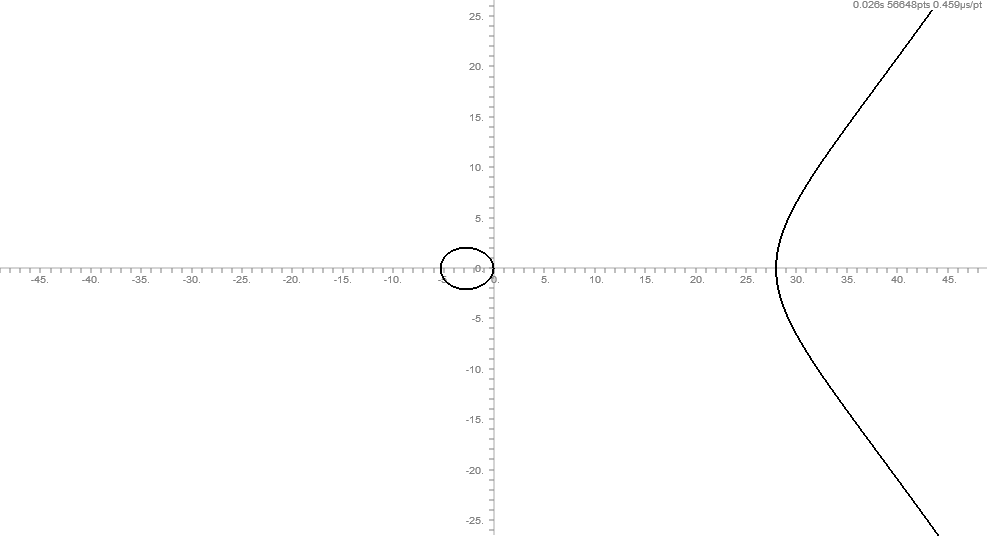}
		\caption{$d=3$, $f=y^2-\frac{2}{100}x^3+\frac{45}{100}x^2+\frac{303}{100}x+\frac{29}{100}$ which has $c=1$ oval and $t=3$ real eigenvectors}
	\label{fig:96}
\end{figure}

\begin{proposizione} \label{prop:100}
Let $c\in\left\{0,1,2 \, nested,2 \, non \, nested,3,4\right\}$ and let $t$ be odd such that $\max(3,2c+1)\leq t\leq13$. Then the set $$\left\{f\in\Sym^4(\R^3)\,|\, f \; has \; c \; ovals,\, \# real \; eigenvectors \; of \; f=t\right\}$$ has positive volume.
\end{proposizione}

\begin{proof}
By Remark \ref{oss:connessione}, we must show examples of ternary quartic forms such that $c\in\left\{0,1,2 \, nested,2 \, non \, nested,3,4\right\}$ and $t$ attains the maximum and the minimum value. We have the following examples:
\begin{itemize}
	\item $t$ maximum. By Remark \ref{oss:drette}, we can take $f=xy(x+y+\frac{1}{3})(-3x+y+1)$, $\epsilon=\frac{1}{1000}$, $g_1=x^4+y^4-1$, $g_2=-x^4-y^4+\frac{5}{2}$, $g_3=7x^4+6y^4-1-5x$ and $g_4=7x^4+6y^4-1-5x-9y$ to obtain $f_1=f+\epsilon g_1$, $f_2=f+\epsilon g_2$, $f_3=f+\epsilon g_3$ and $f_4=f+\epsilon g_4$ with, respectively, $4$, $3$, $2$ non nested and $1$ ovals and $13$ real eigenvectors (see Figures \ref{fig:81}, \ref{fig:82}, \ref{fig:83}, \ref{fig:84}, \ref{fig:87}). Moreover, we can take the hyperbolic quartic $f_5=\det(I+xM_1+yM_2)$, where
\onehalfspacing{$$
M_1=\left(
\begin{array}{cccc}
\frac{2}{9} & 5 & 10 & \frac{7}{4} \\
5 & 1 & 1 & \frac{3}{8} \\
10 & 1 & \frac{1}{2} & \frac{1}{2} \\
\frac{7}{4} & \frac{3}{8} & \frac{1}{2} & \frac{5}{3}
\end{array}
\right), \,
M_2=\left(
\begin{array}{cccc}
\frac{1}{2} & 1 & \frac{1}{2} & \frac{4}{5} \\
1 & 8 & \frac{1}{3} & 8 \\
\frac{1}{2} & \frac{1}{3} & \frac{1}{3} & 8 \\
\frac{4}{5} & 8 & 8 & \frac{7}{8}
\end{array}
\right)
$$}
are symmetric matrices, with $2$ nested ovals and $t=13$ (see Figure \ref{fig:90}) and the Fermat quartic $f_6=x^4+y^4+1$ with $0$ ovals and $t=13$.
  \item $t$ minimum. Then we have:
\begin{itemize}
	\item $f$ has $0$ ovals. In this case, we can find the SOS (sum of squares) form \onehalfspacing{$f=q_1^2+q_2^2+q_3^2=(6x^2+\frac{9}{8}xy+\frac{4}{9}y^2+\frac{1}{6}x+\frac{2}{9}y+\frac{4}{9})^2+(4x^2+\frac{1}{2}xy+\frac{7}{9}y^2+\frac{6}{7}x+\frac{3}{4}y+2)^2+(\frac{7}{3}x^2+\frac{2}{5}xy+\frac{1}{10}y^2+x+\frac{1}{2}y+\frac{1}{5})^2$} with $3$ real eigenvectors.
	\item $f$ has $1$ oval. In this case, we can find the form \onehalfspacing{$f=\frac{9}{5}x^4+\frac{4}{5}x^3y+\frac{1}{3}x^2y^2+\frac{4}{9}xy^3+\frac{5}{4}y^4+x^3+\frac{8}{7}x^2y+\frac{8}{5}xy^2+\frac{1}{5}y^3+x^2+\frac{3}{8}xy+2y^2+\frac{5}{2}x+\frac{5}{9}y+\frac{3}{10}$} (see Figure \ref{fig:88}) with $3$ real eigenvectors.
	\item $f$ has $2$ ovals non nested. In this case, we can find the form $f=q_1q_2=(8x^2+3y^2-\frac{1}{10}xy+3x-10y-9)(7x^2+3y^2+5xy-7x+12y+15)$ (see Figure \ref{fig:86}) with $5$ real eigenvectors.
	\item $f$ has $2$ nested ovals. In this case, we can find the determinantal form $f=\det(I+xN_1+yN_2)$ (see Figure \ref{fig:91}), where 
\onehalfspacing{$$
N_1=\left(
\begin{array}{cccc}
\frac{5}{2} & \frac{5}{3} & 2 & \frac{9}{10} \\
\frac{5}{3} & \frac{7}{2} & \frac{1}{4} & \frac{2}{5} \\
2 & \frac{1}{4} & \frac{10}{7} & \frac{1}{3} \\
\frac{9}{10} & \frac{2}{5} & \frac{1}{3} & 1
\end{array}
\right), \,
N_2=\left(
\begin{array}{cccc}
\frac{4}{5} & \frac{5}{3} & 1 & \frac{5}{8} \\
\frac{5}{3} & \frac{1}{2} & 1 & 1 \\
1 & 1 & 2 & \frac{8}{7} \\
\frac{5}{8} & 1 & \frac{8}{7} & \frac{10}{7}
\end{array}
\right)
$$}
are symmetric matrices, with $5$ real eigenvectors.
  \item $f$ has $3$ ovals. In this case, we have the quartic $f=(x^2+y^2)^2+p(x^2+y^2)+q(x^3-3xy^2)+r$, where $p=\frac{16}{3}$, $q=\frac{80}{9}$, $r=\frac{2624}{9}$ in Figure \ref{fig:85} (see \cite{EC}, pag. 116, 123), with $7$ real eigenvectors.
  \item $f$ has $4$ ovals. In this case, we have the singular form $f=(y^2-\frac{2}{100}x^3+\frac{45}{100}x^2+\frac{303}{100}x+\frac{29}{100})(x-45)$, with $9$ real eigenvectors and then we can perturb $f$ by $\epsilon g$, where $g$ is a quartic such that $f_6=f+\epsilon g$ has $4$ ovals and $\epsilon$ is small enough, to obtain a form with $c=4$ and again $t=9$; we can take $\epsilon=\frac{1}{1000}$ and $g=-x^4-y^4-1$ (see Figures \ref{fig:97}, \ref{fig:98}).
\end{itemize}
\end{itemize}
\end{proof}

\begin{figure}[H]
	\centering
		\includegraphics[width=10cm,height=4cm]{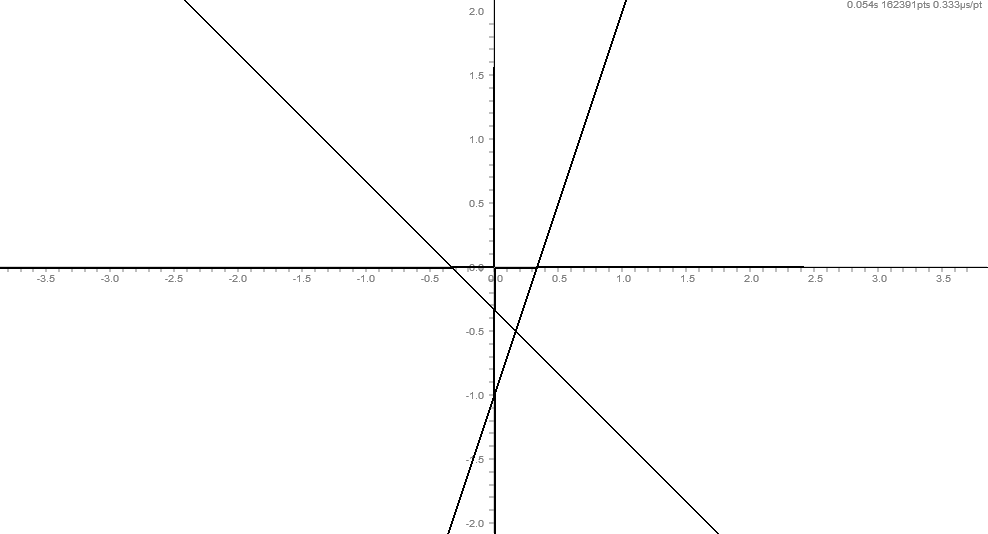}
		\caption{$d=4$, $f=xy(x+y+\frac{1}{3})(-3x+y+1)$}
	\label{fig:81}
\end{figure}
\begin{figure}[H]
	\centering
		\includegraphics[width=10cm,height=6cm]{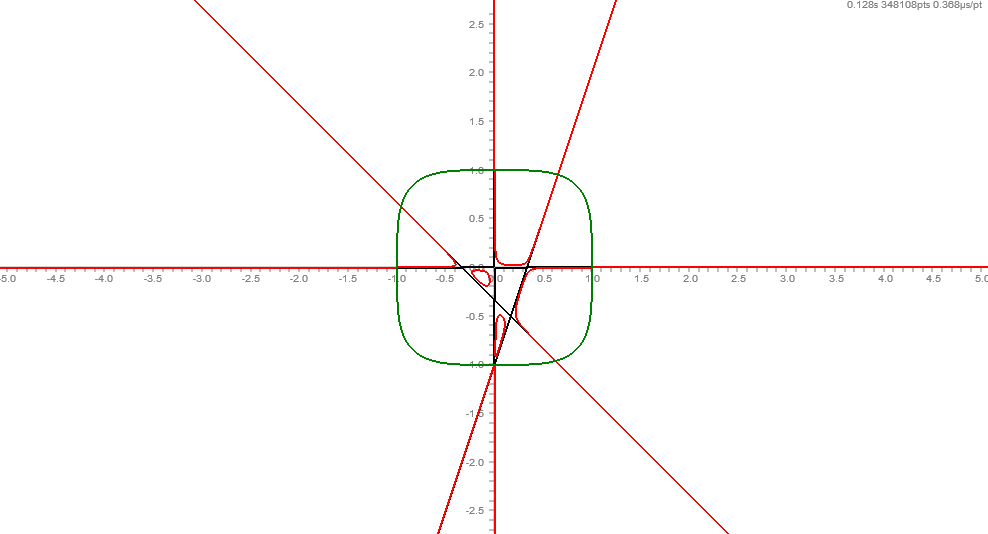}
		\caption{$d=4$, $f=xy(x+y+\frac{1}{3})(-3x+y+1)$, $g_1=x^4+y^4-1$ which is negative on the six singular points of $f$, $f_1=f+\frac{1}{1000}g_1$ which has $4$ ovals}
	\label{fig:82}
\end{figure}
\begin{figure}[H]
	\centering
		\includegraphics[width=10cm,height=6cm]{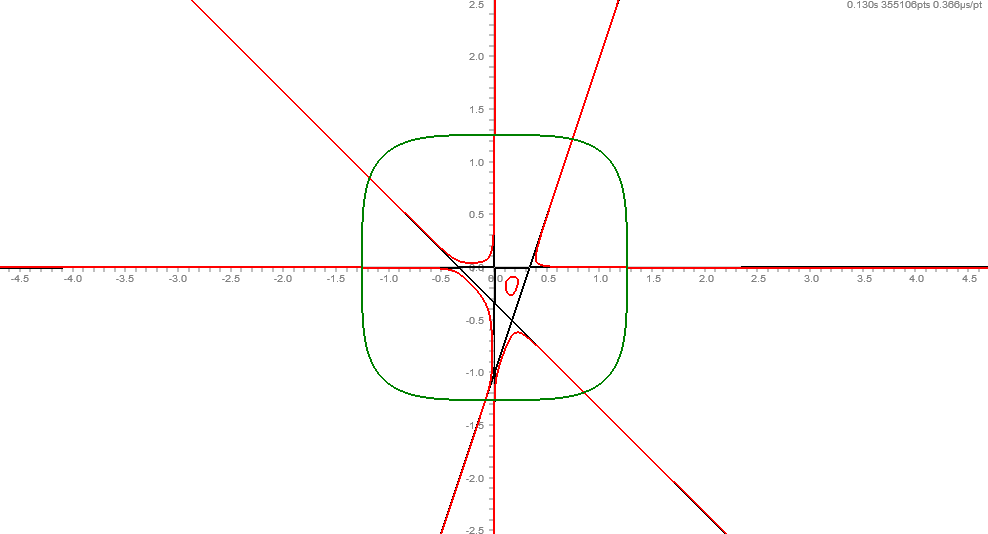}
		\caption{$d=4$, $f=xy(x+y+\frac{1}{3})(-3x+y+1)$, $g_2=-x^4-y^4+\frac{5}{2}$ which is positive on the six singular points of $f$, $f_2=f+\frac{1}{1000}g_2$ which has $3$ ovals}
	\label{fig:83}
\end{figure}
\begin{figure}[H]
	\centering
		\includegraphics[width=10cm,height=6cm]{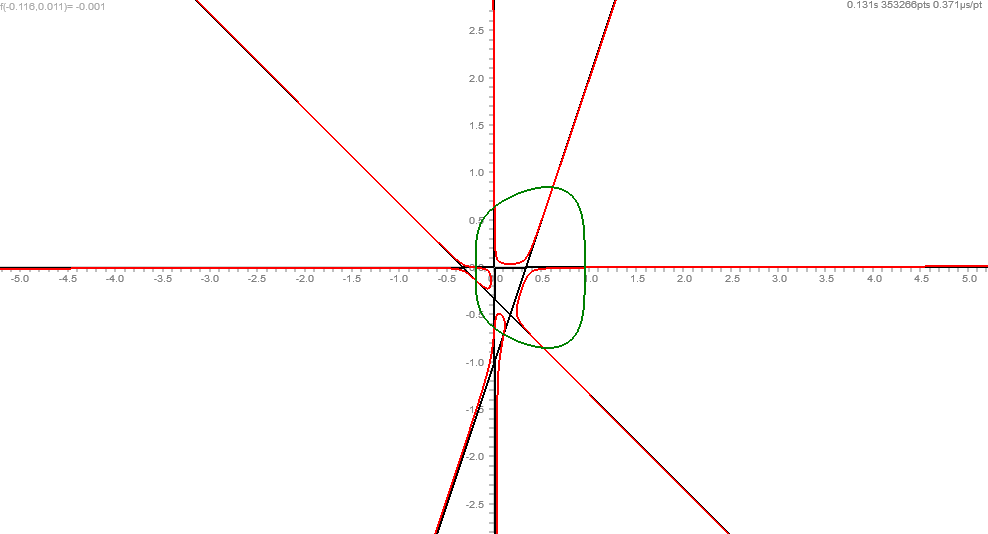}
		\caption{$d=4$, $f=xy(x+y+\frac{1}{3})(-3x+y+1)$, $g_3=7x^4+6y^4-1-5x$ which is negative on four of the six singular points of $f$ and it is positive on the other two, $f_3=f+\frac{1}{1000}g_3$ which has $2$ non nested ovals}
	\label{fig:84}
\end{figure}
\begin{figure}[H]
	\centering
		\includegraphics[width=10cm,height=6cm]{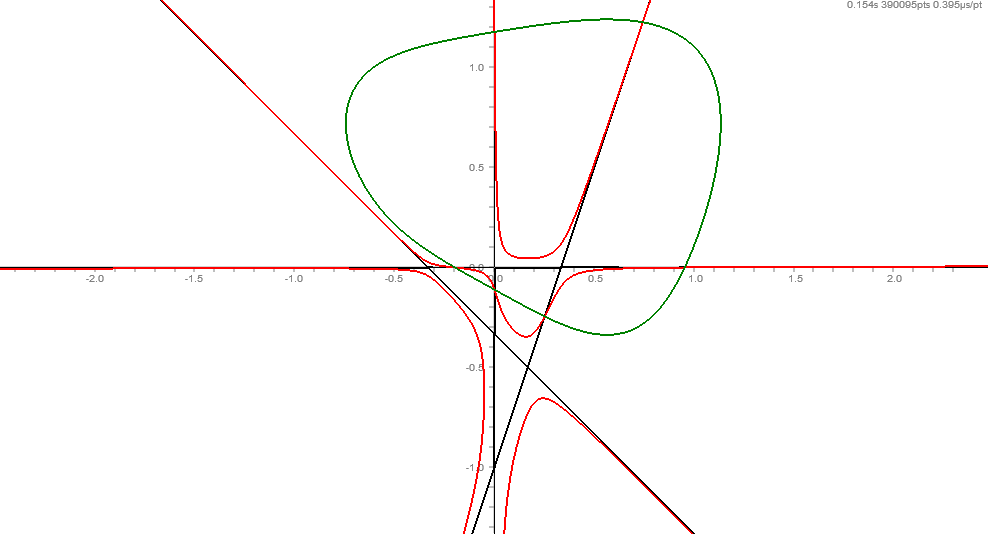}
		\caption{$d=4$, $f=xy(x+y+\frac{1}{3})(-3x+y+1)$, $g_4=7x^4+6y^4-1-5x-9y$ which is positive on four of the six singular points of $f$ and it is negative on the other two, $f_4=f+\frac{1}{1000}g_4$ which has $1$ oval}
	\label{fig:87}
\end{figure}
\begin{figure}[H]
	\centering
		\includegraphics[width=10cm,height=6cm]{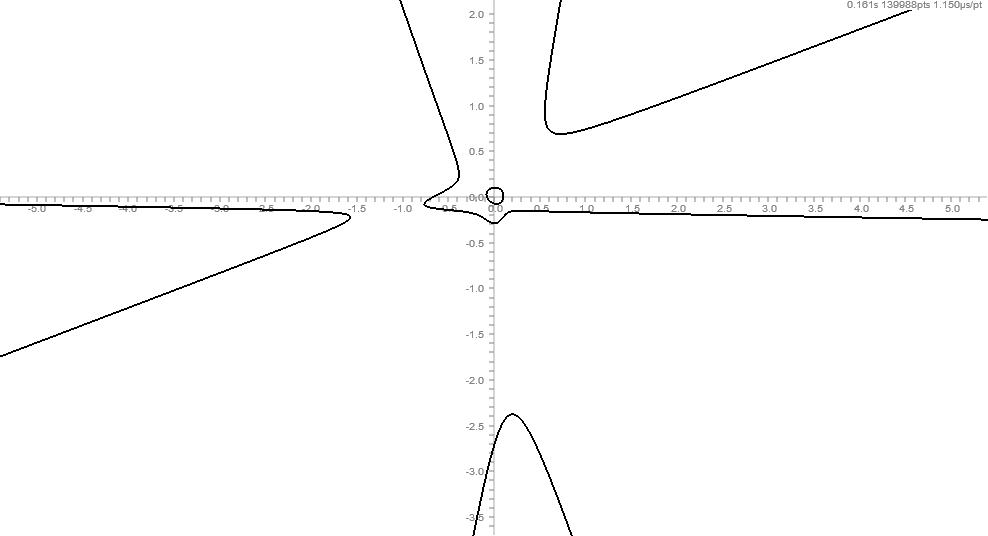}
		\caption{$d=4$, $f_5=\det(I+xM_1+yM_2)$ which has $2$ nested ovals and $13$ real eigenvectors}
	\label{fig:90}
\end{figure}
\begin{figure}[H]
	\centering
		\includegraphics[width=10cm,height=6cm]{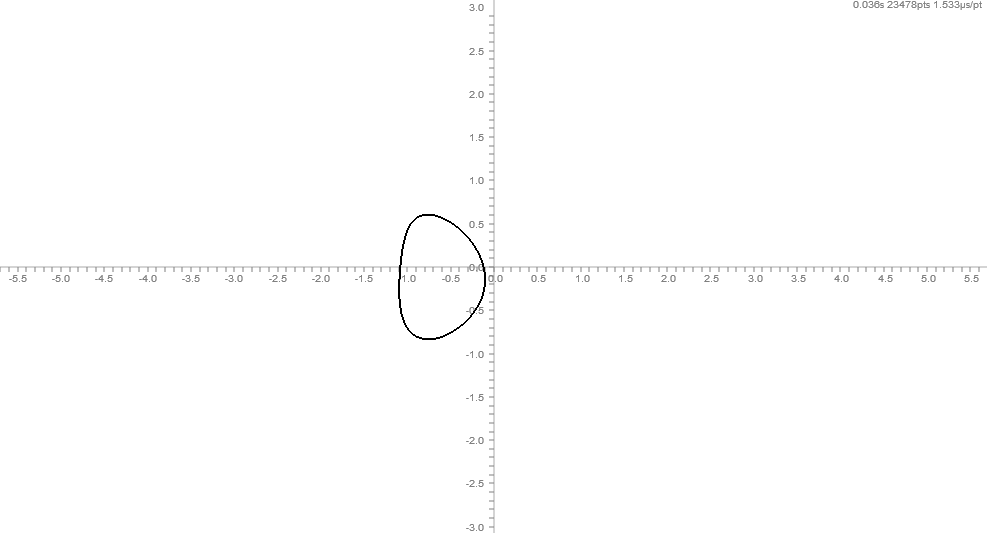}
		\caption{$d=4$, $f=\frac{9}{5}x^4+\frac{4}{5}x^3y+\frac{1}{3}x^2y^2+\frac{4}{9}xy^3+\frac{5}{4}y^4+x^3+\frac{8}{7}x^2y+\frac{8}{5}xy^2+\frac{1}{5}y^3+x^2+\frac{3}{8}xy+2y^2+\frac{5}{2}x+\frac{5}{9}y+\frac{3}{10}$ which has $c=1$ oval and $t=3$ real eigenvectors}
	\label{fig:88}
\end{figure}
\begin{figure}[H]
	\centering
		\includegraphics[width=10cm,height=6cm]{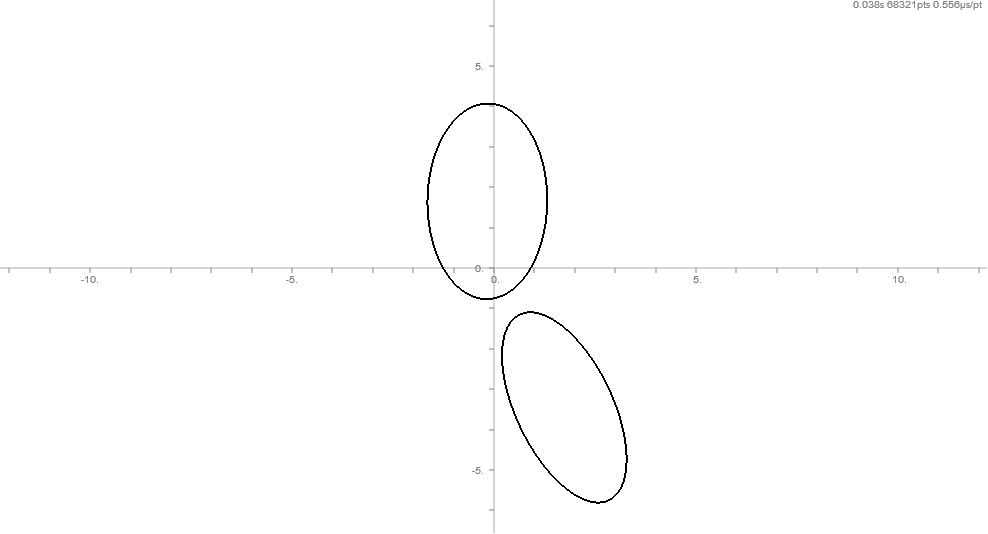}
		\caption{$d=4$, $f=(8x^2+3y^2-\frac{1}{10}xy+3x-10y-9)(7x^2+3y^2+5xy-7x+12y+15)$ which has $c=2$ non nested ovals and $t=5$ real eigenvectors}
	\label{fig:86}
\end{figure}
\begin{figure}[H]
	\centering
		\includegraphics[width=10cm,height=6cm]{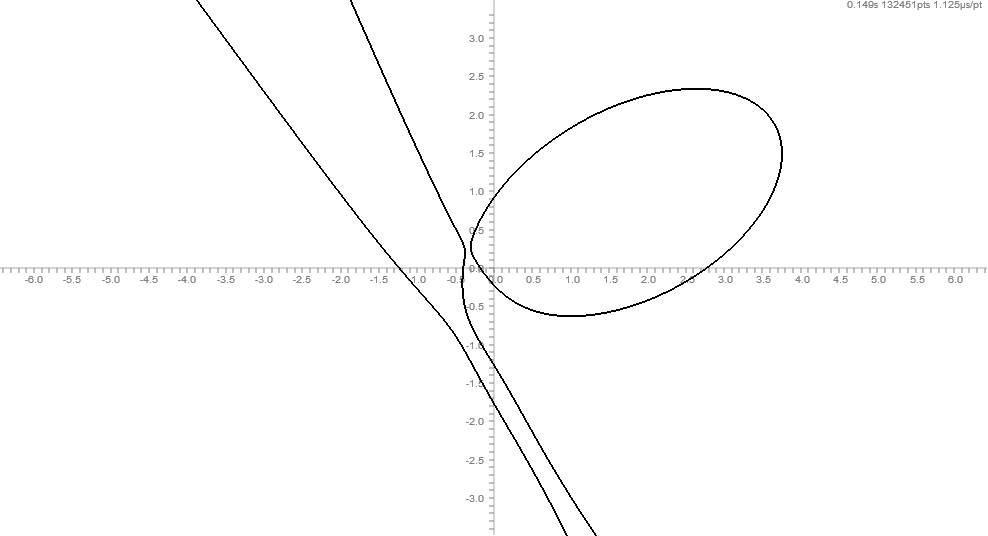}
		\caption{$d=4$, $f=\det(I+xN_1+yN_2)$ which has $c=2$ nested ovals and $t=5$ real eigenvectors}
	\label{fig:91}
\end{figure}
\begin{figure}[H]
	\centering
		\includegraphics[width=10cm,height=6cm]{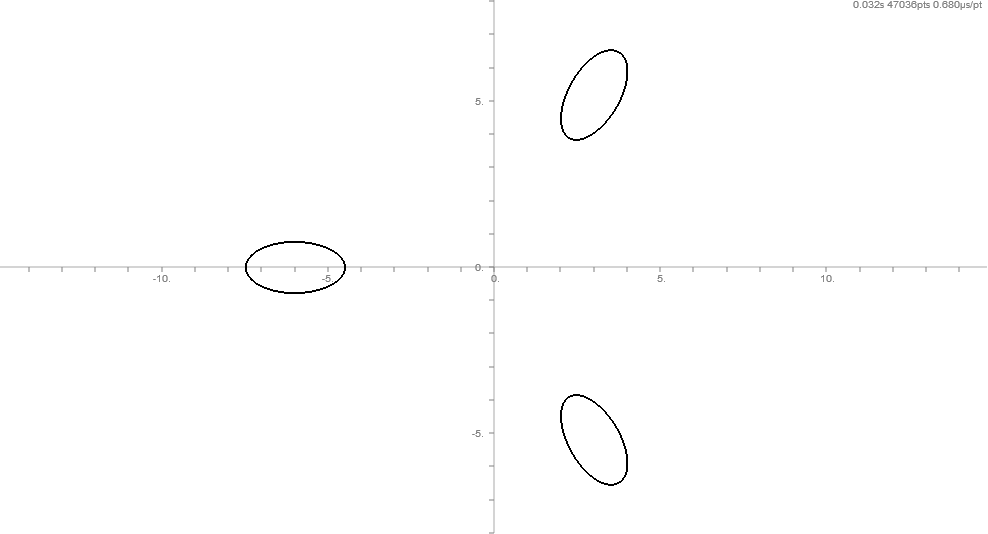}
		\caption{$d=4$, $f=(x^2+y^2)^2+\frac{16}{3}(x^2+y^2)+\frac{80}{9}(x^3-3xy^2)+\frac{2624}{9}$ which has $c=3$ ovals and $t=7$ real eigenvectors}
	\label{fig:85}
\end{figure}
\begin{figure}[H]
	\centering
		\includegraphics[width=10cm,height=4cm]{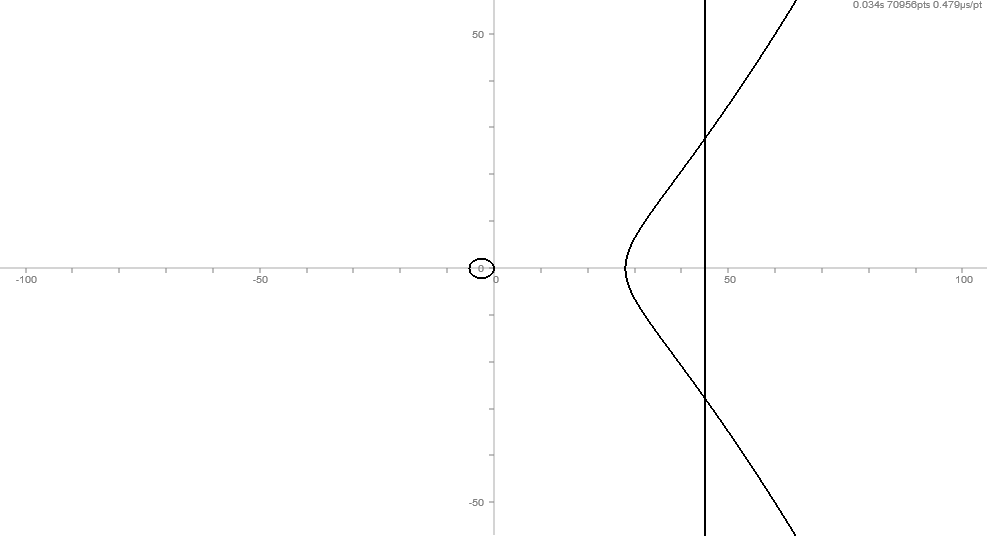}
		\caption{$d=4$, $f=(y^2-\frac{2}{100}x^3+\frac{45}{100}x^2+\frac{303}{100}x+\frac{29}{100})(x-45)$ which has $t=9$ real eigenvectors}
	\label{fig:97}
\end{figure}	
\begin{figure}[H]
	\centering
		\includegraphics[width=10cm,height=4cm]{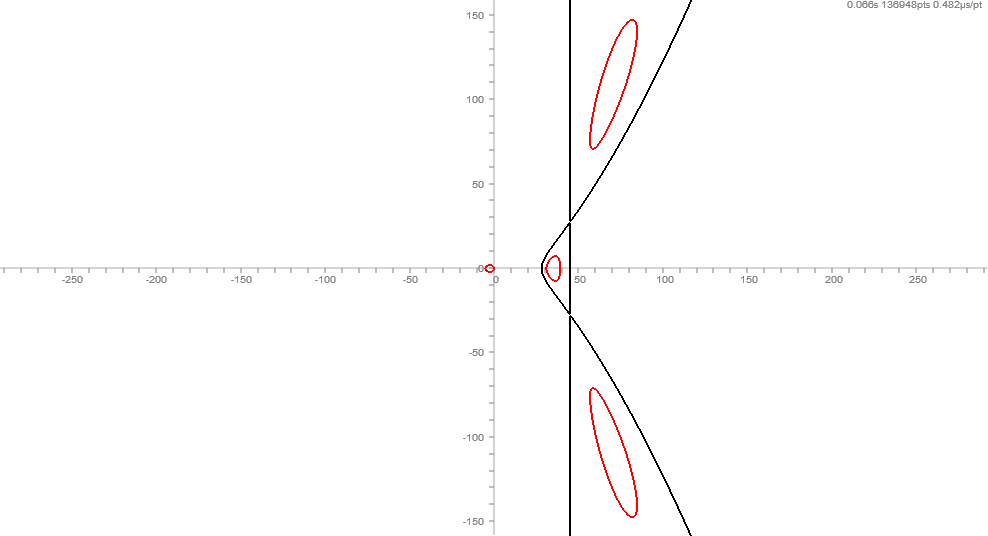}
		\caption{$d=4$, $f=(y^2-\frac{2}{100}x^3+\frac{45}{100}x^2+\frac{303}{100}x+\frac{29}{100})(x-45)$ which has $t=9$ real eigenvectors, $g=-x^4-y^4-1$, $f_6=f+\frac{1}{1000}g$ which has $c=4$ ovals and $t=9$ real eigenvectors}
	\label{fig:98}
\end{figure}

\section{Examples} \label{sez:5}

\begin{osservazione} \em \label{oss:esempi2}
Having fixed the topological type of a ternary quartic $f$, for a sample of $1000$ forms we give the occurrences of all possible values of $t$ in some topological cases:
\begin{enumerate}
	\item $f$ nonnegative, i.e. $c=0$. In this case, we can write $f$ as a sum of squares of $3$ ternary quadratic forms $q_1$, $q_2$, $q_3$ and we have the following table:
\begin{table}[H]
\begin{tabular}{||p{3cm}||*{7}{c|}|} 
\hline
$t$ & $3$ & $5$ & $7$ & $9$ & $11$ & $13$ \\
\hline
\bfseries occurrences & $458$ & $240$ & $215$ & $79$ & $6$ & $2$ \\
\hline
\end{tabular}
\caption{$d=4$}
\end{table}
Note that if $c=0$ all possible numbers of real eigenvectors can occur, also $3$, according to Theorem \ref{teo:4}.
  \item $f$ has one oval, i.e. $c=1$. In this case, we can write $f$ as a product of two quadratic forms $q_1$, $q_2$, where $q_1$ or $q_2$ has empty real locus of zeros and we have the following table:
\begin{table}[H]
\begin{tabular}{||p{3cm}||*{7}{c|}|} 
\hline
$t$ & $3$ & $5$ & $7$ & $9$ & $11$ & $13$ \\
\hline
\bfseries occurrences & $399$ & $397$ & $141$ & $42$ & $16$ & $5$ \\
\hline
\end{tabular}
\caption{$d=4$}
\end{table}
Note that if $c=1$ all possible numbers of real eigenvectors can occur, also $3$, according to Theorem \ref{teo:4}.
  \item $f$ hyperbolic, i.e. $c=2$ and the ovals are nested if $\left\{f=0\right\}$ is smooth in $\PP^2(\C)$. In this case, we can write $f$ as $\det(xI+yM_2+zM_3)$, where $M_i$ are $4\times4$ Hermitian matrices and $I$ is the identity matrix, that is symmetric matrices in this case, because $f$ has real coefficients and we have the following table:
\begin{table}[H]
\begin{tabular}{||p{3cm}||*{7}{c|}|} 
\hline
$t$ & $3$ & $5$ & $7$ & $9$ & $11$ & $13$ \\
\hline
\bfseries occurrences & $0$ & $17$ & $161$ & $315$ & $401$ & $106$ \\
\hline
\end{tabular}
\caption{$d=4$}
\end{table}
Note that if $c=2$ (and the ovals are nested in this case) all possible numbers of real eigenvectors can occur except $3$, according to Theorem \ref{teo:4}.
\end{enumerate}
\end{osservazione}
\begin{osservazione} \label{oss:esempi3} \em
Having fixed the topological type of a ternary sextic $f$, for a sample of $1000$ forms we give the occurrences of all possible values of $t$ in some topological cases:
\begin{enumerate}
  \item $f$ nonnegative, i.e. $c=0$. In this case, we have two possibilities for our form: $f$ is a sum of squares of $4$ ternary cubic forms $q_1$, $q_2$, $q_3$, $q_4$ or not.\\
In the first case, we have the following table:
\begin{table}[H]
\begin{tabular}{||p{3cm}||*{9}{c|}|} 
\hline
$t$ & $3$ & $5$ & $7$ & $9$ & $11$ & $13$ & $15$ & $17$ \\
\hline
\bfseries occurrences & $71$ & $373$ & $33$ & $168$ & $42$ & $11$ & $3$ & $2$ \\
\hline
\end{tabular}
\begin{tabular}{||p{3cm}||*{8}{c|}|} 
\hline
$t$ & $19$ & $21$ & $23$ & $25$ & $27$ & $29$ & $31$ \\
\hline
\bfseries occurrences & $0$ & $0$ & $0$ & $0$ & $0$ & $0$ & $0$ \\
\hline
\end{tabular}
\caption{$d=6$}
\end{table}
In the second case, $f$ is nonnegative but it is not a sum of squares; thus, taking a known sextic with this property, for example $f_1=x^4y^2+x^2y^4+z^6-3x^2y^2z^2$ (the Motzkin's sextic, \cite{R22}), $f_2=x^6+y^6+z^6-x^4y^2-x^2y^4-x^4z^2-y^4z^2-x^2z^4-y^2z^4+3x^2y^2z^2$ (the Robinson's sextic, \cite{R22}), $f_3=x^4y^2+y^4z^2+z^4x^2-3x^2y^2z^2$ (the Choi-Liu's sextic, \cite{R22}), we perturb them without changing their topological type, adding $\epsilon g$, where $g$ is a random SOS sextic. We have the following tables:
\begin{table}[H]
\begin{tabular}{||p{3cm}||*{9}{c|}|} 
\hline
$t$ & $3$ & $5$ & $7$ & $9$ & $11$ & $13$ & $15$ & $17$ \\
\hline
\bfseries occurrences & $0$ & $0$ & $1$ & $11$ & $36$ & $61$ & $200$ & $525$ \\
\hline
\end{tabular}
\begin{tabular}{||p{3cm}||*{8}{c|}|} 
\hline
$t$ & $19$ & $21$ & $23$ & $25$ & $27$ & $29$ & $31$ \\
\hline
\bfseries occurrences & $156$ & $10$ & $0$ & $0$ & $0$ & $0$ & $0$ \\
\hline
\end{tabular}
\caption{$d=6$, $f_1$}
\end{table}
\begin{table}[H]
\begin{tabular}{||p{3cm}||*{9}{c|}|} 
\hline
$t$ & $3$ & $5$ & $7$ & $9$ & $11$ & $13$ & $15$ & $17$ \\
\hline
\bfseries occurrences & $0$ & $0$ & $0$ & $1$ & $2$ & $7$ & $35$ & $28$ \\
\hline
\end{tabular}
\begin{tabular}{||p{3cm}||*{8}{c|}|} 
\hline
$t$ & $19$ & $21$ & $23$ & $25$ & $27$ & $29$ & $31$ \\
\hline
\bfseries occurrences & $47$ & $84$ & $186$ & $610$ & $0$ & $0$ & $0$ \\
\hline
\end{tabular}
\caption{$d=6$, $f_2$}
\end{table}
\begin{table}[H]
\begin{tabular}{||p{3cm}||*{9}{c|}|} 
\hline
$t$ & $3$ & $5$ & $7$ & $9$ & $11$ & $13$ & $15$ & $17$ \\
\hline
\bfseries occurrences & $0$ & $0$ & $0$ & $0$ & $0$ & $0$ & $2$ & $5$  \\
\hline
\end{tabular}
\begin{tabular}{||p{3cm}||*{8}{c|}|} 
\hline
$t$ & $19$ & $21$ & $23$ & $25$ & $27$ & $29$ & $31$ \\
\hline
\bfseries occurrences & $13$ & $20$ & $70$ & $701$ & $173$ & $14$ & $2$ \\
\hline
\end{tabular}
\caption{$d=6$, $f_3$}
\end{table}
  \item $f$ hyperbolic, i.e. $c=3$ and the ovals are nested if $\left\{f=0\right\}$ is smooth in $\PP^2(\C)$. In this case, we can write $f$ as $\det(xI+yM_2+zM_3)$, where $M_i$ are $6\times6$ Hermitian matrices and $I$ is the identity matrix, that is symmetric matrices in this case, because $f$ has real coefficients and we have the following table:
\begin{table}[H]
\begin{tabular}{||p{3cm}||*{9}{c|}|} 
\hline
$t$ & $3$ & $5$ & $7$ & $9$ & $11$ & $13$ & $15$ & $17$ \\
\hline
\bfseries occurrences & $0$ & $0$ & $1$ & $2$ & $11$ & $23$ & $91$ & $174$  \\
\hline
\end{tabular}
\begin{tabular}{||p{3cm}||*{8}{c|}|} 
\hline
$t$ & $19$ & $21$ & $23$ & $25$ & $27$ & $29$ & $31$ \\
\hline
\bfseries occurrences & $261$ & $207$ & $163$ & $49$ & $16$ & $1$ & $1$ \\
\hline
\end{tabular}
\caption{$d=6$}
\end{table}
  \item $f$ is obtained by slightly perturbing six lines, i.e. we perturb the product of six linear forms $l_1,\ldots,l_6$ by adding $\epsilon g$, where $g$ is a random sextic, and we have the following table:
\begin{table}[H]
\begin{tabular}{||p{3cm}||*{9}{c|}|} 
\hline
$t$ & $3$ & $5$ & $7$ & $9$ & $11$ & $13$ & $15$ & $17$ \\
\hline
\bfseries occurrences & $0$ & $0$ & $0$ & $2$ & $7$ & $9$ & $17$ & $35$  \\
\hline
\end{tabular}
\begin{tabular}{||p{3cm}||*{8}{c|}|} 
\hline
$t$ & $19$ & $21$ & $23$ & $25$ & $27$ & $29$ & $31$ \\
\hline
\bfseries occurrences & $49$ & $65$ & $75$ & $97$ & $145$ & $218$ & $281$ \\
\hline
\end{tabular}
\caption{$d=6$}
\end{table}
Note that in this case not all possible numbers of real eigenvectors can occur, because the minimum value of $t$ is not $3$, but $9$.
\end{enumerate}
Then we have the following
\end{osservazione}

\begin{lemmaa}
Let $c\in\left\{0,3 \, nested\right\}$ and let $t$ be odd such that $\max(3,2c+1)\leq t\leq31$. Then the set $$\left\{f\in\Sym^6(\R^3)\,|\, f \; has \; c \; ovals,\, \# real \; eigenvectors \; of \; f=t\right\}$$ has positive volume.
\end{lemmaa}

Dipartimento di Matematica e Informatica "`U. Dini"', Università di Firenze, Viale Morgagni 67 A Firenze, Italy.\\
\textbf{Email address: maccioni@math.unifi.it}
\end{document}